\documentclass[reqno,12pt]{amsart}
\newcommand{\NN}{{\mathbb N}}

\def\bege{\begin{equation}} \def\ende{\end{equation}}
\def\begr{\begin{eqnarray}} \def\endr{\end{eqnarray}}
\usepackage{amsmath}
\usepackage{amssymb}
\usepackage{cite}

\usepackage{bbm}
\usepackage{amsmath}
\usepackage{txfonts}
\usepackage{stmaryrd}
\usepackage{amssymb}
\usepackage{amsfonts}
\usepackage{times}
\usepackage{mathrsfs}

\usepackage{amsthm}
\usepackage{enumerate}

\usepackage{framed}
\usepackage{lipsum}
\usepackage{color}

\newcommand{\DD}{{\mathbb D}}
\def\B{\mathcal{B}}

\def\D{\mathbb{D}}
\def\N{\mathbb N}

\def\a{\alpha}
\def\b{\beta}

\def\p{{\prime}}

\def \tp{AT_p^\infty(\alpha)}
\def \p0{AT_{p,0}^\infty(\alpha)}
\def \tq{AT_q^\infty(\beta)}
\def \q0{AT_{q,0}^\infty(\beta)}
\def\begr{\begin{eqnarray}} \def\endr{\end{eqnarray}}

\def\ol{\overline}
\newtheorem{Lemma}{Lemma}[section]
\newtheorem{Theorem}[Lemma]{Theorem}

\newtheorem{Definition}[Lemma]{Definition}

\newcounter{other}            % Questions get letters
\begin{document}
	\title[]{Generalized integration operators on  Analytic Tent spaces $\tp$}
	
	\author{Rong Yang, Songxiao Li$^\dagger$ and Lian Hu}
	
	\address{Rong Yang
	\\Institute of Fundamental and Frontier Sciences, University of Electronic Science and Technology of China, 610054, Chengdu, Sichuan, P.R. China.}
	\email{yangrong071428@163.com }

	\address{Songxiao Li\\ Department of Mathematics, Shantou University, 515063, Shantou, Guangdong, P.R. China. } \email{jyulsx@163.com}

	\address{Lian Hu\\  Institute of Fundamental and Frontier Sciences, University of Electronic Science and Technology of China, 610054, Chengdu, Sichuan, P.R. China.  }
	\email{hl152808@163.com}


	\begin{abstract} In this paper, we investigate the boundedness and compactness of generalized integration operators $T_g^{n,k}$ and $S_g^{n,0}$ from analytic tent spaces $AT_p^\infty(\a)$ to $\tq$ when $0<p,q<\infty,\a,\b>-2$.

		\thanks{$\dagger$ Corresponding author.}
	\end{abstract}
	
	\maketitle
	
\section{Introduction}
Let $\mathbb{D}$ denote the unit disk on the complex plane $\mathbb{C}$, $\mathbb{T}$ its boundary, and $H(\mathbb{D})$ be the set of
all analytic functions in $\mathbb{D}$. 
Let $\eta \in \mathbb{T}$ and $\zeta>\frac{1}{2}$. The non-tangential region $\Gamma_\zeta(\eta)$ is defined by
$$
\Gamma(\eta)=\Gamma_\zeta(\eta) =\{z \in \mathbb{D}:|z-\eta|<\zeta(1-|z|^2) \}.
$$	
Let $0<p, q<\infty$ and $\a>-2$. The tent space $T_p^q(\alpha)$ consists of all measurable functions $f$ on $\mathbb{D}$ with
\begin{align*}%\label{2.111}
	\|f\|_{T_p^q(\alpha)}=\left(\int_\mathbb{T}\left(\int_{\Gamma(\eta)}|f(z)|^p(1-|z|^2)^{\a} d A(z)\right)^{\frac{q}{p}}|d \eta|\right)^{\frac{1}{q}}<\infty,
\end{align*}
where $dA(z)=\frac{1}{\pi}dxdy$ is the normalized Lebesgue area measure on $\D$. 

For $p=\infty$ and $0<q<\infty$, the tent space $T_{\infty}^q(\alpha)$ consists of all measurable functions $f$ on $\mathbb{D}$ such that
\begin{align*}%\label{2.222}
	\|f\|_{T_\infty^q(\alpha)}=\left(\int_\mathbb{T} \left(\operatorname{esssup}_{z \in \Gamma(\eta)}|f(z)|\right)^q|d \eta|\right)^{\frac{1}{q}}<\infty. 
\end{align*}
It is observed that the definition of \( T^q_\infty(\alpha) \) is independent of \(\alpha\). Thus, we denote it simply as $T^q_\infty$.

For $q=\infty$ and $0<p<\infty$, the tent space $T_p^{\infty}(\alpha)$ consists of all measurable functions $f$ on $\mathbb{D}$ for which
$$
\|f\|_{T_p^{\infty}(\alpha)}=\operatorname{esssup} _{\eta \in \mathbb{T}}\left(\sup _{u \in \Gamma(\eta)} \frac{1}{1-|u|^2} \int_{S(u)}|f(z)|^p(1-|z|^2)^{\a+1} d A(z)\right)^{\frac{1}{p}}<\infty,
$$
where
$$
S(r e^{i \theta})=\left\{\lambda e^{i t}: |t-\theta| \leq \frac{1-r}{2},1-\lambda \leq 1-r\right\}
$$
for $r e^{i \theta} \in \mathbb{D} \backslash\{0\}$ and $S(0)=\mathbb{D}$.
Let $T_{p,0}^\infty(\a)$ be the subspace of $T_p^\infty(\a)$ 
consisting of all measurable functions $f$ with
$$\lim_{|u|\to1}\frac1{1-|u|^2}\int\limits_{S(u)}|f(z)|^p(1-|z|^2)^{\alpha+1}dA(z)=0.$$
Denote $T_p^q(\alpha) \cap H(\mathbb{D})$ by $A T_p^q(\alpha)$. In particular, for $\a=0$, we write $AT_p^q$ instead of $AT_p^q(\a)$.

Tent spaces were initially introduced by Coifman, Meyer, and Stein in  \cite{cms}, serving as a pivotal tool for investigating topics within harmonic analysis. They developed an extensive framework that not only addressed the exploration of significant spaces such as Hardy spaces and Bergman spaces but also provided a unified approach to these inquiries. In the definition outlined above, the aperture $\zeta$ of the non-tangential region $\Gamma_\zeta(\eta)$ is implicitly understood rather than explicitly stated, as it is a well-established fact that varying apertures result in the same function space, each equipped with equivalent quasinorms.
The characterization of the Hardy space through the non-tangential maximal function reveals that $AT^q_\infty$ is equivalent to $H^q$, which can be interpreted as the limit of $AT_p^q(\alpha)$ as $p$ tends towards infinity (refer to \cite{z2} for details). It is particularly noteworthy that $AT_p^p(\alpha - 1)$ corresponds to the weighted Bergman space, denoted as $A^p_\alpha$.

Let $\mathbb{N}$ represent the collection of positive integers. Consider $g \in H(\mathbb{D})$, $k \in \mathbb{N} \cup \{0\}$, and $n \in \mathbb{N}$ with the condition $0 \leq k < n$. The generalized integration operator $T_g^{n, k}$ is defined as follows:
\[
T_g^{n, k} f(z) = I^n\left(f^{(k)}(z) g^{(n-k)}(z)\right), \quad \text{for} \quad f \in H(\mathbb{D}).
\]
In this context, $I^n$ denotes the $n$-th iteration of the integration operator $I$, where $I f(z) = \int_0^z f(t) \, dt$.
The operator $T_g^{n, k}$ was initially presented by Chalmoukis \cite{ch}. Specifically, when $n = 1$ and $k = 0$, the formula simplifies to:
\[
T_g^{1,0} f(z) = \int_0^z f(\eta) g'(\eta) \, d\eta = T_g f(z).
\]

The Volterra integration operator \( T_g \) was first defined and analyzed by Pommerenke in \cite{p}, where it was demonstrated that \( T_g \) is bounded on \( H^2 \) if and only if \( g \) belongs to \( BMOA \). Aleman and Siskakis extended this result in \cite{as1}, proving that \( T_g \) is bounded on \( H^p \) for \( p \geq 1 \) under the same condition that \( g \) is an element of \( BMOA \). Additionally, they established in \cite{as2} that \( T_g \) is bounded on \( A^p \) if and only if \( g \) is within the Bloch space \( \B \). For a deeper investigation into the properties of the operator \( T_g \), refer to \cite{mppw,hyl,lly,llw,w2,wz2,cw} and the related literature.	
	
Chalmoukis, in his study \cite{ch}, delineated the conditions under which the operator \( T_{g}^{n,k} \) is bounded on Hardy spaces. Specifically, he proved that \( T_g^{n,0}: H^p \to H^p \) is bounded if and only if \( g \) is an element of \( BMOA \), and \( T_g^{n,k}: H^p \to H^p \) is bounded if and only if \( g \) belongs to the Bloch space \( \B \) for \( k \geq 1 \). Furthermore, \( T_g^{n,k}: H^p \to H^q \) is bounded if and only if the following supremum is finite:
\[
\sup_{z \in \mathbb{D}} (1 - |z|^2)^{\frac{1}{q} - \frac{1}{p} + n - k} |g^{(n-k)}(z)| < \infty
\]
for \( 0 < p < q < \infty \).
In \cite{dlq}, Du et al. characterized the boundedness, compactness, and Schatten class membership of the operator \( T_g^{n,k} \) on Bergman spaces with doubling weights. These findings highlight the distinct behavior of \( T_g^{n,k} \) in comparison to \( T_g \).
In \cite{yhl} and \cite{yl}, we explored the boundedness and compactness of the generalized integration operator \( T_g^{n,k} \) from \( AT_p^q(\alpha) \) to \( AT_t^s(\beta) \) for \( 0 < p, q, s, t < \infty \), \( \alpha, \beta > -2 \), and from \( AT_\infty^p(H^p) \) to \( AT_\infty^q(H^q) \) for \( 0 < q < p < \infty \), respectively.	

In  \cite{qz}, Qian and Zhu introduced and examined the operator \( S_g^{n, k} \), defined as:
\[
S_g^{n, k} f(z) = I^n\left(f^{(n-k)}(z) g^{(k)}(z)\right).
\]
Specifically, for the case when \( n = 1 \) and \( k = 0 \), we have:
\[
S_g^{1,0} f(z) = \int_0^z f'(\eta) g(\eta) \, d\eta = I_g f(z).
\]
It is evident that \( T_g^{n, k} = S_g^{n, n-k} \) for \( k \neq 0 \). Consequently, the interesting case is \( S_g^{n,0} \).

In this paper, we continue the work from \cite{yhl} and \cite{yl} to study the boundedness and compactness of the operator  $T_g^{n,k}:\tp\to \tq$ and $T_g^{n,k}:\p0\to \q0$ for $0<p,q<\infty$. 
Moreover, we study the boundedness and compactness of the generalized integration operators $S_g^{n,0}:\tp\to\tq$ and $S_g^{n,0}:\p0\to\q0$ for $0<p,q<\infty$.

	The paper is organized as follows. In Section 2, we mainly provide some preliminaries. In Section 3, we describe the boundedness of the operator $T_g^{n,k}$ and  $S_g^{n,0}$. In the final section, we characterize the compactness of the operator $T_g^{n,k}$ and  $S_g^{n,0}$.

	Throughout this paper, the notation \( A \lesssim B \) is used to denote that there exists a positive constant \( C \) such that \( A \leq CB \). Note that the constant \( C \) may differ in various contexts. The notation \( A \asymp B \) signifies that both \( A \lesssim B \) and \( B \lesssim A \) hold true simultaneously.

	\section{Preliminaries}
	In this section, we state some notations and lemmas that will be used in the proof of the main results in this paper.

    A positive Borel measure $\mu$ on $\mathbb{D}$ is said to be a Carleson measure (see \cite{ca1,ca2}) if 
    $$
    \|\mu\|_{CM}=\sup\limits_{I\subseteq\mathbb{T}}\frac{\mu(S(I))}{|I|}<\infty.
    $$
    Here $S(I)$ denotes the Carleson box associated with $I$, i.e.,
    $$
    S(I)=\left\{z\in\D: 1-|I|\le z<1, \frac{z}{|z|}\in I\right\}.
    $$

    A positive Borel measure $\mu$ on $\mathbb{D}$ is called a vanishing Carleson measure if
	$$
	\lim_{|I|\to0}\frac{\mu(S(I))}{|I|}=0.
	$$

	Let $\rho(z,w)=\left| \frac{z-w}{1-\ol{w}z} \right|$ and $\beta(z, w)$ be the hyperbolic metric on $\mathbb{D}$, that is, 
	$$
	\b(z,w)=\frac{1}{2}\log\frac{1+\rho(z,w)}{1-\rho(z,w)},\quad z,w\in\D.
	$$
    Let $D(z, r)=\{w \in \mathbb{D}: \beta(z, w)<r\}$ be the hyperbolic disk of radius $r>0$ centered at $z \in \mathbb{D}$. 
The sequence $Z=\{a_j\}$ is a separated sequence if there is a constant $\tau>0$ with $\beta(a_k, a_j) \geq \tau$ for $k \neq j$. 
For $r>\kappa>0$, the sequence $Z=\{a_j\}$ is said to be an $(r, \kappa)$-lattice if 
$\mathbb{D}=\bigcup_j D(a_j, r)$ and
the sets $D(a_j, \kappa)$ are pairwise disjoint.
It is obvious that any $(r, \kappa)$-lattice is a separated sequence.

\begin{Definition}
Let $Z=\{a_j\}$ be an $(r, \kappa)$-lattice. 
For $0<p<\infty$, the tent space of sequences $T^\infty_p(Z)$ consists of all $\{x_j\}$ for which
$$
\|\{x_j\}\|_{T_p^{\infty}(Z)}=\operatorname{esssup} _{\eta \in \mathbb{T}}\left(\sup _{u \in \Gamma(\eta)} \frac{1}{1-|u|^2} \sum_{a_j \in S(u)}|x_j|^p(1-|a_j|^2)\right)^{\frac{1}{p}}<\infty.
$$
Let \( T_{p,0}^\infty(Z) \) be the subset of \( T_p^{\infty}(Z) \) consisting of all sequences \( \{x_j\} \) such that
$$
\lim\limits_{|u|\to1}\frac{1}{1-|u|^2}\sum\limits_{a_j\in S(u)}|x_j|^p(1-|a_j|^2)=0.
$$
\end{Definition}
	
The aperture $\zeta$ of the non-tangential region $\Gamma_\zeta(\eta)$  is suppressed in the above definition since any two apertures generate the tent spaces of sequences  with equivalent quasinorms.
It is easy to check that the sequence elements \( \{x_j\} \in T_p^\infty(Z) \) if and only if the measure \( d\mu_t = \sum_j |x_j|^p (1 - |a_j|^2) \delta_{a_j} \) is a Carleson measure, where $\delta _{a_j}$ is the Dirac point mass at $a_j$. 
Also, $\{x_j\}\in T_{p,0}^\infty(Z)$ if and only if $d\mu_t$ is a vanishing Carleson measure.

The subsequent lemma provides a discrete description of tent spaces $\tp$, which  establishes the connection between the discrete form of tent spaces and the analytic tent spaces.

\begin{Lemma}\cite[Lemma 3.2]{cw}
	\label{3.2}
Let $0<p<\infty$, $\a>-2$ and $Z=\{a_j\}$ be an $(r,\kappa)$-lattice. Then the following conditions hold.
\begin{enumerate}
\item[(i)] A function $f\in \tp$ if and only if 
$$
\{f(a_j)(1-|a_j|^2)^\frac{\a+2}{p}\}\in T_p^\infty(Z).
$$
Moreover, 
$$
\|f\|_{\tp}\asymp\|\{ f(a_j)(1-|a_j|^2)^\frac{\a+2}{p} \}\|_{T_p^\infty(Z)}.
$$
\item[(ii)] A function $f\in \p0$ if and only if 
$$
\{f(a_j)(1-|a_j|^2)^\frac{\a+2}{p}\}\in T_{p,0}^\infty(Z).
$$
\end{enumerate}
\end{Lemma}

The following lemma is the atomic decomposition of tent spaces $\tp$, which plays a key role in the proof of our main theorem.

\begin{Lemma}\cite[Theorem 3.4]{cw}
	\label{3.4}
	Let $0<p<\infty$, $\a>-2$, $L>\max\{1,\frac{1}{p}\}$ and $Z=\{a_j\}$ be an $(r,\kappa)$-lattice. 
	For any sequence $\{x_j\}\in T_p^\infty(Z)$, the function 
	$$
	f(z)=\sum_{j=1}^{\infty}x_j\frac{(1-|a_j|^2)^L}{(1-\ol{a_j}z)^{L+\frac{\a+2}{p}}}
	$$
	belongs to $\tp$, and $\|f\|_{\tp}\lesssim\|\{x_j\}\|_{T_p^\infty(Z)}$.
\end{Lemma}

The following lemma is the atomic decomposition corresponding to little space \( \p0 \).
\begin{Lemma}\cite[Theorem 3.6]{cw}
	\label{3.6}
	Let $0<p<\infty$, $\a>-2$ and $L>\max\{1,\frac{1}{p}\}$. Then there exists an $(r,\kappa)$-lattice  $Z=\{a_j\}$ such that the space $\p0$ consists exactly of functions of the form
	$$
	f(z)=\sum_{j=1}^{\infty}x_j\frac{(1-|a_j|^2)^L}{(1-\ol{a_j}z)^{L+\frac{\a+2}{p}}}
	$$
	where $\{x_j\}\in T_{p,0}^\infty(Z)$.
\end{Lemma}

Given a sequence $y=\{y_j\}$, the pointwise multiplier $M_y$ is defined by 
$$
M_yx=\{y_jx_j\}
$$
for $x=\{x_j\}$.
The following lemma describes the multipliers of the tent spaces of sequences.

\begin{Lemma}\cite[Lemma 2.2]{cw}
	\label{2.2}
	Let $0<q<p<\infty$ and $Z=\{a_j\}$ be a separated sequence and $\{x_j\}$. Then the following conditions hold.
	\begin{enumerate}
	\item[(i)] $M_y$ maps $T_p^\infty(Z)$ into $T_q^\infty(Z)$ if and only if $\{y_j\}\in T_{\frac{pq}{p-q}}^\infty(Z)$. Moreover, $\|M_y\|=\|y_j\|_{T_{\frac{pq}{p-q}}^\infty}(Z)$. 
		
	\item[(ii)] $M_y$ maps $T_{p,0}^\infty(Z)$ into $T_{q,0}^\infty(Z)$ if and only if $\{y_j\}\in T_{\frac{pq}{p-q}}^\infty(Z)$. 
		
	\item[(iii)] $M_y$ maps $T_p^\infty(Z)$ into $T_{q,0}^\infty(Z)$ if and only if $\{y_j\}\in T_{\frac{pq}{p-q},0}^\infty(Z)$. 
	\end{enumerate}
\end{Lemma}

The following lemma is the Littlewood-Paley formula for tent spaces $\tp$, which plays a crucial role and is used repeatedly in our proof process.

\begin{Lemma}\label{lp}
Let $0<p<\infty$,$\a>-2$ and $n\in\N$.
Then the following conditions hold.
\begin{enumerate}
\item[(i)] A function $f\in\tp$ if and only if for each (or some) $t>0$,
$$
\sup_{a\in\D}\int_{\D}\frac{(1-|a|^2)^t}{|1-\ol{a}z|^{t+1}}|f^{(n)}(z)|^p(1-|z|^2)^{np+\a+1}dA(z)<\infty.
$$

\item[(ii)] A function $f\in\p0$ if and only if $f\in\tp$ and for each (or some) $t>0$,
$$
\lim_{|a|\to 1}\int_{\D}\frac{(1-|a|^2)^t}{|1-\ol{a}z|^{t+1}}|f^{(n)}(z)|^p(1-|z|^2)^{np+\a+1}dA(z)=0.
$$
\end{enumerate}
Moreover, 
\begin{align}\label{norm}
\|f\|_{\tp}^p\asymp\sum_{j=1}^{n-1}|f^{(j)}(0)|^p+
\sup_{a\in\D}\int_{\D}\frac{(1-|a|^2)^t}{|1-\ol{a}z|^{t+1}}|f^{(n)}(z)|^p(1-|z|^2)^{np+\a+1}dA(z).
\end{align}
\end{Lemma}

\begin{proof}
It is easy to check that $f\in AT_p^\infty(\a)$  if and only if the measure $d\mu_f(z)=|f(z)|^p(1-|z|^2)^{\a+1}dA(z)$ is a Carleson measure, and $f\in AT_{p,0}^\infty(\a)$ if and only if $d\mu_f$ is a vanishing Carleson measure.
In addition,  
from \cite[Theorem 3.3]{cw}, we obtain that  $f\in \tp$ if and only if $f^{(n)}\in AT_p^\infty(np+\a)$, and   $f\in \p0$ if and only if $f^{(n)}\in AT_{p,0}^\infty(np+\a)$. 
Furthermore, by \cite[Theorem 45]{z2},  $\mu$ is a (vanishing) Carleson measure if and only if, for each (or some) $t>0$, 
$$
\sup\limits_{a\in\mathbb{D}}\int\limits_{\mathbb{D}}\frac{(1-|a|^2)^t}{|1-\ol{a}z|^{t+1}}d\mu(z)<\infty\left(\quad
\lim\limits_{|a|\to1^-}\int\limits_{\mathbb{D}}\frac{(1-|a|^2)^t}{|1-\ol{a}z|^{t+1}}d\mu(z)=0\right). $$
Then, we see that {\it (i)} and {\it (ii)} hold.
Moreover, using \cite[Theorem 3.3]{cw}, we
have
 $$
\|f\|_{\tp}\asymp\sum_{j=1}^{n-1}|f^{(j)}(0)|+\|f^{(n)}\|_{AT_p^\infty(np+\a)},
$$ which implies $(\ref{norm})$.
\end{proof}

The following lemma provides a description of the growth properties of functions in tent spaces $\tp$.

\begin{Lemma}\label{z}
	Let $0<p<\infty$, $\a>-2$, $n\in\N$ and $f\in\tp$. Then 
	$$
	|f^{(n)}(z)|\lesssim\frac{\|f\|_{\tp}}{(1-|z|^2)^{\frac{\a+2}{p}+n}}
	$$
	for all $z\in\D$.
\end{Lemma}
\begin{proof}
	Using \cite[Proposition 4.13]{z1} and Lemma \ref{lp}, for any $t>0$, we obtain
	\begin{align*}
		|f^{(n)}(z)|^p&\lesssim\frac{1}{(1-|z|^2)^{\a+2}}\int_{D(z,\frac{1}{2})}|f^{(n)}(w)|^p(1-|w|^2)^\a dA(w)\\
		&\asymp\frac{1}{(1-|z|^2)^{\a+2+np}}\int_{D(z,\frac{1}{2})}\frac{(1-|z|^2)^t}{|1-\ol{z}w|^{t+1}}|f^{(n)}(w)|^p(1-|w|^2)^{np+\a+1} dA(w)\\
		&\lesssim\frac{\|f\|^p_{\tp}}{(1-|z|^2)^{\a+2+np}},
	\end{align*}
	which yields the desired result.
\end{proof}

\begin{Lemma}\cite[Lemma 2.5]{of}\label{2.5}
	Let  $s>-1$, $r,t>0$ and $r+t-s-2>0$.  
		If $r<s+2<t$, then 
		$$
		\int_\D\frac{(1-|z|^2)^s}{|1-\bar{a}z|^{r}|1-\bar{b}z|^t}dA(z)\lesssim\frac{1}{|1-\ol{a}b|^r(1-|b|^2)^{t-s-2}}
		$$
		for all $z\in\D$.
\end{Lemma}

\section{Boundedness}
In this section, we describe the boundedness of the generalized integration operators $T_g^{n,k}:\tp\to\tq$ and  $T_g^{n,k}:\p0\to\q0$ for $0<p,q<\infty$.
Furthermore, we study the boundedness of the operators $S_g^{n,0}:\tp\to\tq$ and $S_g^{n,0}:\p0\to\q0$ for $0<p, q<\infty$.
Now we state and prove the main results in this section.

\begin{Theorem}\label{th1}
Let $0<p\le q<\infty$, $\a,\b>-2$, $g\in H(\D)$, $n\in\N$ and $k\in\N\cup\{0\}$ such that $0 \leq k<n$.
Then $T_g^{n,k}:AT_p^\infty(\a)\to AT_q^\infty(\b)$ is bounded if and only if 
$$
U_g=\sup_{z\in\D}|g^{(n-k)}(z)|(1-|z|^2)^{n-k+\frac{\b+2}{q}-\frac{\a+2}{p}}<\infty.
$$
\end{Theorem}
\begin{proof}

{\bf Sufficiency.}	
Assume that $U_g<\infty$.
For $f\in\tp$ and $t>0$, using Lemmas \ref{lp} and \ref{z}, we obtain
\begin{equation*}\label{4.1}
\begin{aligned}
\|T_g^{n,k}f\|_{\tq}^q
\asymp&\sup_{a\in\D}\int_{\D}\frac{(1-|a|^2)^t}{|1-\ol{a}z|^{t+1}}|(T_g^{n,k}f)^{(n)}(z)|^q(1-|z|^2)^{nq+\b+1}dA(z)\\
	=&\sup_{a\in\D}\int_{\D}\frac{(1-|a|^2)^t}{|1-\ol{a}z|^{t+1}}|f^{(k)}(z)|^q|g^{(n-k)}(z)|^q(1-|z|^2)^{nq+\b+1}dA(z)\\
	\lesssim& \sup_{z\in\D}|g^{(n-k)}(z)|^q(1-|z|^2)^{\left(n-k+\frac{\b+2}{q}-\frac{\a+2}{p}\right)q}\\
	&\cdot\sup_{a\in\D}\int_{\D}\frac{(1-|a|^2)^t}{|1-\ol{a}z|^{t+1}}|f^{(k)}(z)|^q(1-|z|^2)^{kq-1+\frac{q(\a+2)}{p}}dA(z)\\
	\lesssim&U_g^q\|f\|^{q-p}_{\tp}\sup_{a\in\D}\int_{\D}\frac{(1-|a|^2)^t}{|1-\ol{a}z|^{t+1}}|f^{(k)}(z)|^p(1-|z|^2)^{kp+\a+1}dA(z)\\
	\lesssim&U_g^q \|f\|^{q}_{\tp},
\end{aligned}
\end{equation*}
which implies the desired result.
	
{\bf Necessity.}
Suppose that $T_g^{n,k}:AT_p^\infty(\a)\to AT_q^\infty(\b)$ is bounded. For $u\in\D$, set
$$
f_u(z)=\frac{1}{(1-\ol{u}z)^\frac{\a+2}{p}}, \quad z\in\D.
$$
For $t>\a+2$, using Lemma \ref{2.5} we get
\begin{align*}
\|f_u\|^p_{\tp}\asymp&\sup_{a\in\D}\int_{\D}\frac{(1-|a|^2)^t}{|1-\ol{a}z|^{t+1}}\frac{(1-|z|^2)^{\a+1}}{|1-\ol{u}z|^{\a+2}}dA(w)\\
\lesssim&\sup_{a\in\D}\frac{(1-|a|^2)^{\a+2}}{|1-\ol{a}u|^{\a+2}}\lesssim1.
\end{align*}
Hence, for $t>0$, using Lemma \ref{lp} and \cite[Proposition 4.13]{z1}, we have
\begin{align*}
\infty>&\|T_g^{n,k}f_u\|_{\tq}^q\asymp\sup_{a\in\D}\int_{\D}\frac{(1-|a|^2)^t}{|1-\ol{a}z|^{t+1}}|(T_g^{n,k}f_u)^{(n)}(z)|^q(1-|z|^2)^{nq+\b+1}dA(z)\\
\ge&\int_{\D}\frac{(1-|u|^2)^t}{|1-\ol{u}z|^{t+1}}|f_u^{(k)}(z)|^q|g^{(n-k)}(z)|^q(1-|z|^2)^{nq+\b+1}dA(z)\\
\gtrsim&\int_{D(u,\frac{1}{2})}\frac{(1-|u|^2)^t(1-|z|^2)^{nq+\b+1}}{|1-\ol{u}z|^{t+1+\frac{q(\a+2)}{p}+kq}}|g^{(n-k)}(z)|^qdA(z)\\
\asymp&\int_{D(u,\frac{1}{2})}(1-|z|^2)^{\b+(n-k)q-\frac{q(\a+2)}{p}}|g^{(n-k)}(z)|^qdA(z)\\
\gtrsim&(1-|u|^2)^{\b+2+(n-k)q-\frac{q(\a+2)}{p}}|g^{(n-k)}(u)|^q,
\end{align*} 
which implies the desired result.
The proof is complete.
\end{proof}

Noting that $f_u\in\p0$ for $u\in\D$, similarly to the proof of Theorem \ref{th1}, we get the following result. We omit the details of the proof.

\begin{Theorem}\label{th10}
Let $0<p\le q<\infty$, $\a,\b>-2$, $g\in H(\D)$, $n\in\N$ and $k\in\N\cup\{0\}$ such that $0 \leq k<n$.
Then $T_g^{n,k}:AT_{p,0}^\infty(\a)\to AT_{q,0}^\infty(\b)$ is bounded if and only if 
$$
U_g=\sup_{z\in\D}|g^{(n-k)}(z)|(1-|z|^2)^{n-k+\frac{\b+2}{q}-\frac{\a+2}{p}}<\infty.
$$
\end{Theorem}

\begin{Theorem}\label{th2}
Let $0<q<p<\infty$, $\a,\b>-2$ with $\frac{p\b-q\a}{p-q}>-2$, $g\in H(\D)$, $n\in\N$ and $k\in\N\cup\{0\}$ such that $0 \leq k<n$.
Then $T_g^{n,k}:AT_{p}^\infty(\a)\to AT_{q}^\infty(\b)$ is bounded if and only if
$$
g^{(n-k)}\in AT^\infty_{\frac{pq}{p-q}}\left(\lambda\right),
$$
where $\lambda=\frac{p\b-q\a}{p-q}$.
\end{Theorem}
\begin{proof}
{\bf Sufficiency.}
For $f\in AT_p^\infty(\a)$
 and $t>0$, using Lemma \ref{lp} and H\"older's inequality, we obtain
\begin{equation}\label{4.2}
\begin{aligned}
\|T_g^{n,k}f\|_{\tq}^q\asymp
&\sup_{a\in\D}\int_{\D}\frac{(1-|a|^2)^t}{|1-\ol{a}z|^{t+1}}|(T_g^{n,k}f)^{(n)}(z)|^q(1-|z|^2)^{nq+\b+1}dA(z)\\
=&\sup_{a\in\D}\int_{\D}\frac{(1-|a|^2)^t}{|1-\ol{a}z|^{t+1}}|f^{(k)}(z)|^q|g^{(n-k)}(z)|^q(1-|z|^2)^{nq+\b+1}dA(z)\\
\le&\sup_{a\in\D}\left( \int_{\D}\frac{(1-|a|^2)^t}{|1-\ol{a}z|^{t+1}}|f^{(k)}(z)|^p(1-|z|^2)^{kp+\a+1}dA(z)  \right)^{\frac{q}{p}}\\
&\cdot\sup_{a\in\D}\left( \int_{\D}\frac{(1-|a|^2)^t}{|1-\ol{a}z|^{t+1}}|g^{(n-k)}(z)|^{\frac{pq}{p-q}}(1-|z|^2)^{(n-k)\frac{pq}{p-q}+\frac{p\b-q\a}{p-q}+1}   \right)^{\frac{p-q}{p}}\\
\lesssim&\|f\|^q_{\tp}\|g^{(n-k)}\|^q_{AT^\infty_{\frac{pq}{p-q}}\left(\lambda\right)}.
\end{aligned}
\end{equation}
Therefore, $T_g^{n,k}:AT_{p}^\infty(\a)\to AT_{q}^\infty(\b)$ is bounded.

{\bf Necessity.}
Let $Z=\{a_j\}$ be an $(r,\kappa)$-lattice and $\{x_j\}\in T_p^\infty(Z)$. Let $r_j:[0,1]\to\{-1,1\}$ be the Radermacher functions. For $L>\max\{1,\frac{1}{p}\}$, set
\begin{align}\label{4.3}
F_u(z)=\sum_{j=1}^{\infty}x_jr_j(u)\frac{(1-|a_j|^2)^L}{|1-\ol{a_j}z|^{L+\frac{\a+2}{p}}},\quad z\in\D.
\end{align}
Using Lemma \ref{3.4}, we have $\|F_u\|_{\tp}\lesssim\|\{x_j\}\|_{T_p^\infty(Z)}$.
For $t>0$, by the assumption and Lemma \ref{lp}, we get
\begin{align*}
	\|T_g^{n,k}F_u\|_{\tq}^q\asymp
&\sup_{a\in\D}\int_{\D}\frac{(1-|a|^2)^t}{|1-\ol{a}z|^{t+1}}|(T_g^{n,k}F_u)^{(n)}(z)|^q(1-|z|^2)^{nq+\b+1}dA(z)\\
=&\sup_{a\in\D}\int_{\D}\frac{(1-|a|^2)^t}{|1-\ol{a}z|^{t+1}}|F_u^{(k)}(z)|^q|g^{(n-k)}(z)|^q(1-|z|^2)^{nq+\b+1}dA(z)\\
\lesssim&\|T_g^{n,k}\|^q\|F_u\|^q_{\tp}\lesssim\|T_g^{n,k}\|^q\|\{x_j\}\|^q_{T_p^\infty(Z)}.
\end{align*}
Integrating with respect to $u$ on $[0,1]$ and using Fubini's theorem and Khinchine's inequality (see \cite{l3}), we obtain
\begin{align*}
&\sup_{a\in\D}\int_{\D}\frac{(1-|a|^2)^t}{|1-\ol{a}z|^{t+1}}\left(  \sum_{j=1}^{\infty}|x_j|^2\frac{(1-|a_j|^2)^{2L}}{|1-\ol{a_j}z|^{2L+2k+\frac{2(\a+2)}{p}}}  \right)^{\frac{q}{2}}|g^{(n-k)}(z)|^q(1-|z|^2)^{nq+\b+1}dA(z)\\
\lesssim&\|T_g^{n,k}\|^q\|\{x_j\}\|^q_{T_p^\infty(Z)}.
\end{align*}
Since each point $z\in\D$ belongs to at most $\N$ of the sets $D(a_j,\rho)$, we have
\begin{align*}
\left(  \sum_{j=1}^{\infty}|x_j|^2\frac{(1-|a_j|^2)^{2L}}{|1-\ol{a_j}z|^{2L+2k+\frac{2(\a+2)}{p}}}  \right)^{\frac{q}{2}}\ge&\left(  \sum_{j=1}^{\infty}|x_j|^2\frac{(1-|a_j|^2)^{2L}\chi_{D(a_j,\rho)}(z)}{|1-\ol{a_j}z|^{2L+2k+\frac{2(\a+2)}{p}}}  \right)^{\frac{q}{2}}\\
\ge&\sum_{j=1}^{\infty}|x_j|^q\frac{\chi_{D(a_j,\rho)}(z)}{(1-|a_j|^2)^{kq+\frac{q(\a+2)}{p}}}.
\end{align*}
Hence,
\begin{align*}
&\sup_{a\in\D}\sum_{j=1}^{\infty}\frac{(1-|a|^2)^t}{|1-\ol{a_j}a|^{t+1}}|x_j|^q(1-|a_j|^2)^{nq-kq+\b+1-\frac{q(\a+2)}{p}}\int_{D(a_j,\rho)}|g^{(n-k)}(z)|^qdA(z)\\
\lesssim&\|T_g^{n,k}\|^q\|\{x_j\}\|^q_{T_p^\infty(Z)}.
\end{align*}
By \cite[Proposition 4.13]{z1}, we get
\begin{align*}
&\sup_{a\in\D}\sum_{j=1}^{\infty}\frac{(1-|a|^2)^t}{|1-\ol{a_j}a|^{t+1}}|x_j|^q(1-|a_j|^2)^{2+nq-kq+\b+1-\frac{q(\a+2)}{p}}|g^{(n-k)}(a_j)|^q\\
\lesssim&\|T_g^{n,k}\|^q\|\{x_j\}\|^q_{T_p^\infty(Z)},
\end{align*}
which implies that
\begin{align*}
\{x_jg^{(n-k)}(a_j)(1-|a_j|^2)^{n-k+\frac{\b+2}{q}-\frac{\a+2}{p}}\}\in T_q^\infty(Z)
\end{align*}
and 
\begin{align*}
\|\{  x_jg^{(n-k)}(a_j)(1-|a_j|^2)^{n-k+\frac{\b+2}{q}-\frac{\a+2}{p}} \}\|_{T_q^\infty(Z)}\lesssim\|T_g^{n,k}\|\|\{x_j\}\|_{T_q^\infty(Z)}.
\end{align*}
Since $\{x_j\}\in T_p^\infty(Z)$ is arbitrary, using Lemma \ref{2.2} $(i)$, we have
\begin{align*}
\{g^{(n-k)}(a_j)(1-|a_j|^2)^{n-k+\frac{\b+2}{q}-\frac{\a+2}{p}}\}\in T_{\frac{pq}{p-q}}^\infty(Z)
\end{align*}
and
\begin{align*}
\| \{g^{(n-k)}(a_j)(1-|a_j|^2)^{n-k+\frac{\b+2}{q}-\frac{\a+2}{p}}\} \|_{T_{\frac{pq}{p-q}}^\infty(Z)}\lesssim\|T_g^{n,k}\|.
\end{align*}
Combining this result with Lemmas \ref{3.2} and \ref{lp}, we can deduce that $g^{(n-k)}\in AT^\infty_{\frac{pq}{p-q}}\left(\lambda\right)$.
The proof is complete.
\end{proof}

\begin{Theorem}\label{th20}
	Let $0<q<p<\infty$, $\a,\b>-2$ with $\frac{p\b-q\a}{p-q}>-2$, $g\in H(\D)$, $n\in\N$ and $k\in\N\cup\{0\}$ such that $0 \leq k<n$.
	Then $T_g^{n,k}:AT_{p,0}^\infty(\a)\to AT_{q,0}^\infty(\b)$ is bounded if and only if
	$$
	g^{(n-k)}\in AT^\infty_{\frac{pq}{p-q}}\left(\lambda\right),
	$$
	where $\lambda=\frac{p\b-q\a}{p-q}$.
\end{Theorem}
\begin{proof}
{\bf Sufficiency.}
It is evident that Lemma \ref{lp} and $(\ref{4.2})$ imply that $T_g^{n,k}:AT_{p,0}^\infty(\a)\to AT_{q,0}^\infty(\b)$ is bounded.

{\bf Necessity.}
Assume that
$T_g^{n,k}:AT_{p,0}^\infty(\a)\to AT_{q,0}^\infty(\b)$ is bounded.
First, we define \( F_u \) as in $(\ref{4.3})$ for the sequence \( \{x_j\} \) that belongs to \( T_{p,0}^\infty(Z) \). Subsequently, it follows from Lemma \ref{3.6} that \( F_u \in \p0 \).
Fix $\epsilon>0$, for $t>0$, using Lemma \ref{lp}, there exists $r\in(0,1)$ such that 
\begin{align*}
\int_{\D}\frac{(1-|a|^2)^t}{|1-\ol{a}z|^{t+1}}|F_u^{(k)}(z)|^q|g^{(n-k)}(z)|^q(1-|z|^2)^{nq+\b+1}dA(z)<\epsilon^q
\end{align*}
whenever $r<|a|<1$. By following the same steps as in the proof of necessity in Theorem \ref{th2}, we obtain
\begin{align*}
\sum_{j=1}^{\infty}\frac{(1-|a|^2)^t}{|1-\ol{a_j}a|^{t+1}}|x_j|^q(1-|a_j|^2)^{2+nq-kq+\b+1-\frac{q(\a+2)}{p}}|g^{(n-k)}(a_j)|^q\lesssim\epsilon^q
\end{align*}
whenever $r<|a|<1$, which implies
\begin{align*}
	\{x_jg^{(n-k)}(a_j)(1-|a_j|^2)^{n-k+\frac{\b+2}{q}-\frac{\a+2}{p}}\}\in T_{q,0}^\infty(Z)
\end{align*}
Using Lemma \ref{2.2} $(ii)$, we have
\begin{align*}
	\{g^{(n-k)}(a_j)(1-|a_j|^2)^{n-k+\frac{\b+2}{q}-\frac{\a+2}{p}}\}\in T_{\frac{pq}{p-q}}^\infty(Z).
\end{align*}
Applying Lemma \ref{3.2}, we obtain
$g^{(n-k)}\in AT^\infty_{\frac{pq}{p-q}}\left(\lambda\right)$.
The proof is complete.
\end{proof}

Next, we investigate the boundedness of the operator $S_g^{n,0}$ from $\tp$ to $\tq$ and from $\p0$ to $\q0$ when \(0 < p,q < \infty\). 
Since the proofs for operators $T_g^{n,k}$ and $S_g^{n,0}$ are similar, we omit the proofs here and only present the conclusions.

\begin{Theorem}\label{sn1}
Let $0<p\le q<\infty$, $\a,\b>-2$ with $\frac{p\b-q\a}{p-q}>-2$, $g\in H(\D)$ and $n\in\NN$. Then the following statements are equivalent.
\begin{enumerate}
\item[(i)] $S_g^{n,0}:\tp\to\tq$ is bounded.

\item[(ii)] $S_g^{n,0}:\p0\to\q0$ is bounded.

\item[(iii)] $
\sup_{z\in\D}|g(z)|(1-|z|^2)^{\frac{\b+2}{q}-\frac{\a+2}{p}}<\infty.
$
\end{enumerate}
\end{Theorem}

\begin{Theorem}\label{sn2}
	Let $0<q<p<\infty$, $\a,\b>-2$ with $\frac{p\b-q\a}{p-q}>-2$, $g\in H(\D)$ and $n\in\NN$. Then the following statements are equivalent.
	\begin{enumerate}
	\item[(i)] $S_g^{n,0}:\tp\to\tq$ is bounded.
		
	\item[(ii)] $S_g^{n,0}:\p0\to\q0$ is bounded.
		
	\item[(iii)] 
	$
	g\in AT^\infty_{\frac{pq}{p-q}}(\frac{p\b-q\a}{p-q}).
	$
	\end{enumerate}
\end{Theorem}

\section{Compactness}

In this section, we establish the compactness of the generalized integration operators $T_g^{n,k}:\tp\to\tq$ and $T_g^{n,k}:\p0\to\q0$ for $0<p,q<\infty$. Moreover, we study the compactness of the operators $S_g^{n,0}:\tp\to\tq$ and $S_g^{n,0}:\p0\to\q0$ for $0<p, q<\infty$.
For this purpose, we require the following lemmas.

\begin{Lemma}\label{le3}
Let $0<p,q<\infty$, $\a,\b>-2$, $g\in H(\D)$, $n\in\N$ and $k\in\N\cup\{0\}$ such that $0 \leq k<n$ and $T_g^{n,k}:AT_{p}^\infty(\a)\to AT_{q}^\infty(\b)$ is bounded.  Then $T_g^{n,k}(H^\infty)\subset\q0$.
\end{Lemma}
\begin{proof}
Assume that $f\in H^\infty$. 
In the case $0<p\le q<\infty$, according to Theorem \ref{th1}, the boundedness of $T_g^{n,k}:AT_p^\infty(\a)\to AT_q^\infty(\b)$ implies that
$U_g=\sup_{z\in\D}|g^{(n-k)}(z)|(1-|z|^2)^{n-k+\frac{\b+2}{q}-\frac{\a+2}{p}}<\infty$.
If \( n - k + \frac{\beta + 2}{q} - \frac{\alpha + 2}{p} < 0 \), then \( T_g^{n,k} = 0 \), and hence, there is no need for further proof.
Now suppose that $n-k+\frac{\b+2}{q}-\frac{\a+2}{p}\ge0$. For sufficiently large $t>0$, using \cite[Lemma 3.10]{z1} we obtain
\begin{align*}
&\lim_{|a|\to 1}\int_{\D}\frac{(1-|a|^2)^t}{|1-\ol{a}z|^{t+1}}|(T_g^{n,k}f)^{(n)}(z)|^q(1-|z|^2)^{nq+\b+1}dA(z)\\
=&\lim_{|a|\to 1}\int_{\D}\frac{(1-|a|^2)^t}{|1-\ol{a}z|^{t+1}}|f^{(k)}(z)|^q|g^{(n-k)}(z)|^q(1-|z|^2)^{nq+\b+1}dA(z)\\
\lesssim&\sup_{z\in\D}|f^{(k)}(z)|^q(1-|z|^2)^{kq}\sup_{z\in\D}|g^{(n-k)}(z)|^q(1-|z|^2)^{\left(n-k+\frac{\b+2}{q}-\frac{\a+2}{p}\right)q}\\
&\cdot\lim_{|a|\to 1}\int_{\D}\frac{(1-|a|^2)^t(1-|z|^2)^{\frac{q(\a+2)}{p}-1}}{|1-\ol{a}z|^{t+1}}dA(z)\\
\lesssim&\|f\|_{H^\infty}^q U_g^q\lim_{|a|\to 1}(1-|a|^2)^{\frac{q(\a+2)}{p}}=0.
\end{align*}
Therefore, $T_g^{n,k}f\in\q0$ by Lemma \ref{lp}.

In the case $0<q<p<\infty$, for sufficiently large $t>0$, using Theorem \ref{th2}, H\"older's inequality and \cite[Lemma 3.10]{z1}, we get
\begin{align*}
&\lim_{|a|\to 1}\int_{\D}\frac{(1-|a|^2)^t}{|1-\ol{a}z|^{t+1}}|(T_g^{n,k}f)^{(n)}(z)|^q(1-|z|^2)^{nq+\b+1}dA(z)\\
=&\lim_{|a|\to 1}\int_{\D}\frac{(1-|a|^2)^t}{|1-\ol{a}z|^{t+1}}|f^{(k)}(z)|^q|g^{(n-k)}(z)|^q(1-|z|^2)^{nq+\b+1}dA(z)\\
\lesssim&\sup_{z\in\D}|f^{(k)}(z)|^q(1-|z|^2)^{kq}\lim_{|a|\to 1}\int_{\D}\frac{(1-|a|^2)^t}{|1-\ol{a}z|^{t+1}}|g^{(n-k)}(z)|^q(1-|z|^2)^{nq-kq+\b+1}dA(z)\\
\le&\|f\|_{H^\infty}^q \sup_{a\in\D}\left( \int_{\D}\frac{(1-|a|^2)^t}{|1-\ol{a}z|^{t+1}}|g^{(n-k)}(z)|^{\frac{pq}{p-q}}(1-|z|^2)^{(n-k)\frac{pq}{p-q}+\frac{p\b-q\a}{p-q}+1}dA(z) \right)^{\frac{p-q}{p}}\\
&\cdot\lim_{|a|\to 1}\left( \int_{\D}\frac{(1-|a|^2)^t}{|1-\ol{a}z|^{t+1}}(1-|z|^2)^{\a+1}dA(z)  \right)^{\frac{q}{p}}
\end{align*}

\begin{align*}
\le&\|f\|_{H^\infty}^q \|g^{(n-k)}\|^q_{AT^\infty_{\frac{pq}{p-q}}\left(\frac{p\b-q\a}{p-q}\right)}\lim_{|a|\to 1}(1-|a|^2)^{\frac{q(\a+2)}{p}}=0.
\end{align*}
Hence, $T_g^{n,k}f\in\q0$ by Lemma \ref{lp}. 
The proof is complete.
\end{proof}

Following the proof of Lemma \ref{le3}, we can derive the following lemma.

\begin{Lemma}\label{le3s}
Let $0<p, q<\infty$, $\a,\b>-2$ with $\frac{p\b-q\a}{p-q}>-2$, $g\in H(\D)$ and $n\in\NN$ such that $S_g^{n,0}:AT_{p}^\infty(\a)\to AT_{q}^\infty(\b)$ is bounded.  
Then $S_g^{n,0}(H^\infty)\subset\q0$.
\end{Lemma}

\begin{Lemma}\label{le4}
Let $0<p,q<\infty$, $\a,\b>-2$ with $\frac{p\b-q\a}{p-q}>-2$, $g\in H(\D)$, $n\in\N$ and $k\in\N\cup\{0\}$ such that $0 \leq k<n$.
If $T_g^{n,k}:AT_{p}^\infty(\a)\to AT_{q}^\infty(\b)$ is compact or $T_g^{n,k}:AT_{p,0}^\infty(\a)\to AT_{q,0}^\infty(\b)$ is compact, then $T_g^{n,k}(\tp)\subset\q0$.
\end{Lemma}
\begin{proof}
{\it (i)} Assume that $f\in \tp$. For $r\in(0,1)$, using Lemma \ref{le3}, we have $T_g^{n,k}f_r\in\q0$. Here $f_r$ is defined by $f_r(z)=f(rz)$ for $z\in\D$. By \cite[Corollary 3.8]{cw},  we see that 
\[ \|f_r\|_{AT_p^\infty(\alpha)} \lesssim \|f\|_{AT_p^\infty(\alpha)}. \] 
Given that \( f_r \) converges uniformly to \( f \) on compact subsets of \( \mathbb{D} \) as \( r \to 1 \), the compactness of the operator \( T_g^{n,k}: AT_p^\infty(\alpha) \to AT_q^\infty(\beta) \) implies the existence of a sequence \( \{r_j\} \) approaching 1 such that \( T_g^{n,k} f_{r_j} \) converges to \( T_g^{n,k} f \) in \( AT_q^\infty(\beta) \). Consequently, \( T_g^{n,k} f \in AT_{q,0}^\infty(\beta) \).

{\it (ii)} Assume that \( f \in AT_p^\infty(\alpha) \) and let \( \{r_j\} \subset (0,1) \) be a sequence converging to 1. Given that \( f_{r_j} \in H^\infty \subset AT_{p,0}^\infty(\alpha) \) and \( \|f_{r_j}\|_{AT_p^\infty(\alpha)} \lesssim \|f\|_{AT_p^\infty(\alpha)} \), the compactness of the operator \( T_g^{n,k}: AT_{p,0}^\infty(\alpha) \to AT_{q,0}^\infty(\beta) \) ensures the existence of a subsequence \( \{r_{j_m}\} \) and an element \( h \in AT_{q,0}^\infty(\beta) \) such that \( T_g^{n,k} f_{r_{j_m}} \to h \) in \( AT_q^\infty(\beta) \), which, by Lemma \ref{z}, also implies uniform convergence on compact subsets of \( \mathbb{D} \). Therefore, \( f_{r_{j_m}}^{(k)} g^{(n-k)} \to h^{(n)} \) uniformly on compact subsets of \( \mathbb{D} \). On the other hand, since \( f_{r_{j_m}} \to f \) uniformly on compact subsets of \( \mathbb{D} \), it follows that \( f^{(k)} g^{(n-k)} = h^{(n)} \), and thus  $(T_g^{n,k} f-h)^{(n)}=0$. Therefore, $T_g^{n,k}f(z)=h(z)+p(z)$, where \( p(z) \) is a polynomial of degree not higher than \( n-1 \) and $p\in\q0$. Hence, $T_g^{n,k}f\in\q0$.
\end{proof}

By Lemma \ref{le3s} and via a proof analogous to Lemma \ref{le4}, we establish the following lemma. The proof is omitted.

\begin{Lemma}\label{le4s}
Let $0<p, q<\infty$, $\a,\b>-2$ with $\frac{p\b-q\a}{p-q}>-2$, $g\in H(\D)$ and $n\in\NN$.
If $S_g^{n,0}:AT_{p}^\infty(\a)\to AT_{q}^\infty(\b)$ is compact or $S_g^{n,0}:AT_{p,0}^\infty(\a)\to AT_{q,0}^\infty(\b)$ is compact, then $S_g^{n,0}(\tp)\subset\q0$.
\end{Lemma}

The next lemma can be proved in a standard way (see \cite[Proposition 3.11]{cm}).

\begin{Lemma}
Let X and Y be two Banach spaces, and T be a bounded linear operator.
Then $T:X\to Y$ is compact if and only if 
$T:X\to Y$ is bounded, and for any bounded sequence $\{f_j\}$ in $X$ which converges to zero uniformly on compact subsets of $\mathbb{D}$, we have $Tf_j\to0$ in $Y$ as $j\to\infty$.
\end{Lemma}

We also need the following lemma, which can be found in \cite{ll}.

\begin{Lemma}\label{2.1}
	For $0<r<1$, let $\chi_{\{z:|z|<r\}}$ be the characteristic function of the set $\{z:|z|<r\}$. If $\mu$ is a Carleson measure on $\mathbb{D}$,
	then $\mu$ is a vanishing Carleson measure if and only if $\|\mu-\mu_{r}\|\rightarrow 0$ as $r\rightarrow 1^{-}$, where $d\mu_{r}=\chi_{\{z:|z|<r\}}d\mu$.
\end{Lemma}

Now, we can state and prove the main results of this section.

\begin{Theorem}\label{th3}
	Let $0<p\le q<\infty$, $\a,\b>-2$, $g\in H(\D)$, $n\in\N$ and $k\in\N\cup\{0\}$ such that $0 \leq k<n$.
	Then $T_g^{n,k}:AT_p^\infty(\a)\to AT_q^\infty(\b)$ is compact if and only if 
	\begin{align*}
	\lim_{|z|\to1}|g^{(n-k)}(z)|(1-|z|^2)^{n-k+\frac{\b+2}{q}-\frac{\a+2}{p}}=0.
	\end{align*}
\end{Theorem}
\begin{proof}
{\bf Sufficiency.}
Suppose that $\lim_{|z|\to1}|g^{(n-k)}(z)|(1-|z|^2)^{n-k+\frac{\b+2}{q}-\frac{\a+2}{p}}=0$ holds. For $\epsilon>0$, there exists $r\in(0,1)$ such that
$$
|g^{(n-k)}(z)|(1-|z|^2)^{n-k+\frac{\b+2}{q}-\frac{\a+2}{p}}<\epsilon
$$
whenever $r\le|z|<1$. 
Consider a bounded sequence $\{f_j\} \subset AT_p^\infty(\alpha)$ that converges to 0 uniformly on compact subsets of $\mathbb{D}$. 
There exists $M>0$ such that $j>M$, we have $|f_j^{(k)}(z)|<\epsilon$ for $z\in r\D$. Consequently, for $t>\frac{q(\a+2)}{p}+kq$, using Lemma \ref{lp}, \cite[Lemma 3.10]{z1}, and Lemma \ref{z}, we derive that
\begin{align*}
&\|T_g^{n,k}f_j\|^q_{\tq}\\
\asymp&\sup_{a\in\D}\left( \int_{r\D}+\int_{\D\backslash r\D}  \right)\frac{(1-|a|^2)^t}{|1-\ol{a}z|^{t+1}}|f^{(k)}_j(z)|^q|g^{(n-k)}(z)|^q(1-|z|^2)^{nq+\b+1}dA(z)\\
\le&\epsilon^q\sup_{z\in\D}|g^{(n-k)}(z)|^q(1-|z|^2)^{\left(n-k+\frac{\b+2}{q}-\frac{\a+2}{p}\right)q}\sup_{a\in\D}\int_{\D}\frac{(1-|a|^2)^t}{|1-\ol{a}z|^{t+1}}(1-|z|^2)^{kq-1+\frac{q(\a+2)}{p}}dA(z)\\
&+\epsilon^q\sup_{a\in\D}\int_{\D}\frac{(1-|a|^2)^t}{|1-\ol{a}z|^{t+1}}|f^{(k)}_j(z)|^q(1-|z|^2)^{kq-1+\frac{q(\a+2)}{p}}dA(z)\\
\lesssim& \epsilon^qU_g^q+\epsilon^q
\|f_j\|^{q-p}_{\tp}\sup_{a\in\D}\int_{\D}\frac{(1-|a|^2)^t}{|1-\ol{a}z|^{t+1}}|f_j^{(k)}(z)|^p(1-|z|^2)^{kp+\a+1}dA(z)\\
\lesssim&\epsilon^q\left(  U_g^q   +\sup_j \|f_j\|_{\tp}^q  \right)
\end{align*}
whenever $j>M$. Here $U_g$ is defined in Theorem \ref{th1}. The arbitrariness of $\epsilon$ implies $T_g^{n,k}:AT_p^\infty(\a)\to AT_q^\infty(\b)$ is compact.

{\bf Necessity.}
Suppose that $T_g^{n,k}:AT_p^\infty(\a)\to AT_q^\infty(\b)$ is compact. 
Let $\{z_j\}$ be a sequence with $\lim_{j\to\infty}|z_j|=1$. Set 
$$
f_j(z)=\frac{(1-|z_j|^2)^{\frac{1}{2p}}}{(1-\ol{z_j}z)^{\frac{\a+2}{p}+\frac{1}{2p}}},\quad z\in\D.
$$
Then, \(f_j \in \tp\), with \(\|f_j\|_{\tp} \lesssim 1\), and \(f_j\) converges uniformly to 0 on compact subsets of \(\mathbb{D}\) as \(j \to \infty\). By the assumption, it follows that
$\|T_g^{n,k}f_j\|_{\tq}\to 0$ as $j\to\infty$.
Similar to the proof of Theorem \ref{th1}, we deduce that
$$
\|T_g^{n,k}f_j\|_{\tq}\gtrsim|g^{(n-k)}(z_j)|(1-|z_j|^2)^{n-k+\frac{\b+2}{q}-\frac{\a+2}{p}},
$$
which implies the desired result.
\end{proof}

\begin{Theorem}\label{th30}
	Let $0<p\le q<\infty$, $\a,\b>-2$, $g\in H(\D)$, $n\in\N$ and $k\in\N\cup\{0\}$ such that $0 \leq k<n$.
	Then $T_g^{n,k}:AT_{p,0}^\infty(\a)\to AT_{q,0}^\infty(\b)$ is compact if and only if 
	\begin{align*}
		\lim_{|z|\to1}|g^{(n-k)}(z)|(1-|z|^2)^{n-k+\frac{\b+2}{q}-\frac{\a+2}{p}}=0.
	\end{align*}
\end{Theorem}
\begin{proof}
{\bf Sufficiency.}
By the assumption, we see that $T_g^{n,k}:AT_{p,0}^\infty(\a)\to AT_{q,0}^\infty(\b)$ is bounded according to Theorem \ref{th10} and $T_g^{n,k}:AT_{p}^\infty(\a)\to AT_{q}^\infty(\b)$ is compact by Theorem \ref{th3}. Hence, $T_g^{n,k}:AT_{p,0}^\infty(\a)\to AT_{q,0}^\infty(\b)$ is compact.

{\bf Necessity.}
The proof of the necessity is conducted in a manner analogous to that employed in the proof of necessity for Theorem \ref{th3}.
The proof is complete.

\end{proof}

\begin{Theorem}\label{th4}
	Let $0<q<p<\infty$, $\a,\b>-2$ with $\frac{p\b-q\a}{p-q}>-2$, $g\in H(\D)$, $n\in\N$ and $k\in\N\cup\{0\}$ such that $0 \leq k<n$.
	Then $T_g^{n,k}:AT_{p}^\infty(\a)\to AT_{q}^\infty(\b)$ is compact if and only if 
	\begin{align*}
	g^{(n-k)}\in AT^\infty_{\frac{pq}{p-q},0}(\lambda),
	\end{align*}
where $\lambda=\frac{p\b-q\a}{p-q}$.
\end{Theorem}
\begin{proof}
{\bf Sufficiency.}
Assume that $g^{(n-k)}\in AT^\infty_{\frac{pq}{p-q},0}\left(\lambda\right)$ holds. Fix $\epsilon>0$.  According to Lemma \ref{2.1}, there exists $r\in(0,1)$ such that 
$$
\sup_{a\in\D}\int_{\D\backslash r\D}\frac{(1-|a|^2)^t}{|1-\ol{a}z|^{t+1}}|g^{(n-k)}(z)|^{\frac{pq}{p-q}}(1-|z|^2)^{(n-k)\frac{pq}{p-q}+\frac{p\b-q\a}{p-q}+1}dA(z)<\epsilon^{\frac{pq}{p-q}}
$$
for sufficiently large $t>0$.
Consider the bounded sequence \(\{f_j\} \subset AT_p^\infty(\alpha)\) that converges to 0 uniformly on compact subsets of \(\mathbb{D}\). We aim to show that \(\|T_g^{n,k}f_j\|_{AT_q^\infty(\beta)} \to 0\) as $j\to\infty$. There exists an \(M > 0\) such that for all \(j > M\), we have \(|f_j^{(k)}(z)| < \epsilon\) for \(z \in r\mathbb{D}\).
Subsequently, employing H\"older's inequality along with \cite[Lemma 3.10]{z1}, we derive that
\begin{align*}
&\|T_g^{n,k}f_j\|^q_{\tq}\\
\asymp&\sup_{a\in\D}\left( \int_{r\D}+\int_{\D\backslash r\D}  \right)\frac{(1-|a|^2)^t}{|1-\ol{a}z|^{t+1}}|f^{(k)}_j(z)|^q|g^{(n-k)}(z)|^q(1-|z|^2)^{nq+\b+1}dA(z)\\
\lesssim&\epsilon^q\left(\sup_{a\in\D}  \int_{r\D}\frac{(1-|a|^2)^t}{|1-\ol{a}z|^{t+1}}|g^{(n-k)}(z)|^{\frac{pq}{p-q}}(1-|z|^2)^{(n-k)\frac{pq}{p-q}+\frac{p\b-q\a}{p-q}+1}dA(z) \right)^{\frac{p-q}{p}}\\
&\cdot\left( \sup_{a\in\D}  \int_{r\D}\frac{(1-|a|^2)^t}{|1-\ol{a}z|^{t+1}}(1-|z|^2)^{kp+\a+1}dA(z) \right)^{\frac{q}{p}}
\end{align*}

\begin{align*}
&+\left( \sup_{a\in\D}  \int_{\D\backslash  r\D}\frac{(1-|a|^2)^t}{|1-\ol{a}z|^{t+1}}|f_j^{(k)}(z)|^p(1-|z|^2)^{kp+\a+1}dA(z) \right)^{\frac{q}{p}}\\
&\cdot\left(\sup_{a\in\D}  \int_{\D\backslash r\D}\frac{(1-|a|^2)^t}{|1-\ol{a}z|^{t+1}}|g^{(n-k)}(z)|^{\frac{pq}{p-q}}(1-|z|^2)^{(n-k)\frac{pq}{p-q}+\frac{p\b-q\a}{p-q}+1}dA(z) \right)^{\frac{p-q}{p}}\\
\lesssim&\epsilon^q\left(\|g^{(n-k)}\|^q_{AT^\infty_{\frac{pq}{p-q}}\left(\lambda\right)} +\sup_{j}\|f_j\|^q_{AT_p^\infty(\a)}   \right)
\end{align*}
whenever $j>M$ and $t>\a+2+kp$. Given that \(\epsilon > 0\) is arbitrary, we deduce that \(\|T_g^{n,k}f_j\|_{AT_q^\infty(\beta)} \to 0\) as $j\to\infty$, which implies that the operator \(T_g^{n,k}: AT_p^\infty(\alpha) \to AT_q^\infty(\beta)\) is compact.

{\bf Necessity.}
Let \(\epsilon > 0\). If the operator \(T_g^{n,k}: AT_p^\infty(\alpha) \to AT_q^\infty(\beta)\) is compact, then by Lemma \ref{le4}, we have $T_g^{n,k}(\tp)\subset\q0$. Define \(F_u\) as (\ref{4.3}) for a sequence \(\{x_j\} \in T_p^\infty(Z)\). Since $T_g^{n,k}(F_u)\in \q0$, there exists a \(r \in (0,1)\) such that
\begin{align*}
\int_{\D}\frac{(1-|a|^2)^t}{|1-\ol{a}z|^{t+1}}|F_u^{(k)}(z)|^q|g^{(n-k)}(z)|^q(1-|z|^2)^{nq+\b+1}dA(z)<\epsilon^q
\end{align*}
whenever $r<|a|<1$. Similar to the proof of Theorem \ref{th2}, we have
\begin{align*}
\{x_jg^{(n-k)}(a_j)(1-|a_j|^2)^{n-k+\frac{\b+2}{q}-\frac{\a+2}{p}}\}\in T_{q,0}^\infty(Z).
\end{align*}
According to Lemma \ref{2.2} $(iii)$, we deduce that
\begin{align*}
	\{g^{(n-k)}(a_j)(1-|a_j|^2)^{n-k+\frac{\b+2}{q}-\frac{\a+2}{p}}\}\in T_{\frac{pq}{p-q},0}^\infty(Z).
\end{align*}
Combining this with Lemmas \ref{3.2} and \ref{lp}, we get that
\begin{align*}
	g^{(n-k)}\in AT^\infty_{\frac{pq}{p-q},0}\left(\lambda\right).
\end{align*}
The proof is complete.

\end{proof}

\begin{Theorem}\label{th40}
	Let $0<q<p<\infty$, $\a,\b>-2$ with $\frac{p\b-q\a}{p-q}>-2$, $g\in H(\D)$, $n\in\N$ and $k\in\N\cup\{0\}$ such that $0 \leq k<n$.
	Then $T_g^{n,k}:AT_{p,0}^\infty(\a)\to AT_{q,0}^\infty(\b)$ is compact if and only if 
	\begin{align*}
	g^{(n-k)}\in AT^\infty_{\frac{pq}{p-q},0}\left(\lambda\right),
	\end{align*}
where $\lambda=\frac{p\b-q\a}{p-q}$.
\end{Theorem}

\begin{proof}
{\bf Sufficiency.}
By the assumption, the operator $T_g^{n,k}:AT_{p,0}^\infty(\a)\to AT_{q,0}^\infty(\b)$ is bounded according to Theorem \ref{th20} and $T_g^{n,k}:AT_{p}^\infty(\a)\to AT_{q}^\infty(\b)$ is compact by Theorem \ref{th4}. Therefore, it follows that $T_g^{n,k}:AT_{p,0}^\infty(\a)\to AT_{q,0}^\infty(\b)$ is compact.

{\bf Necessity.}
The proof of necessity is similar to the proof of necessity in Theorem \ref{th4}.
We omit the detail of the proof.
The proof is complete.

\end{proof}

Subsequently, we examine the compactness of the operator \( S_g^{n,0} \) from \( \tp \) to \( \tq \) and from \( \p0 \) to \( \q0 \) when \( 0 < p,q < \infty \).
Since the proofs for operators $T_g^{n,k}$ and $S_g^{n,0}$ are similar, we omit the proofs here and only present the conclusions.

\begin{Theorem}\label{sn3}
	Let $0<p\le q<\infty$, $\a,\b>-2$ with $\frac{p\b-q\a}{p-q}>-2$, $g\in H(\D)$ and $n\in\NN$. Then the following statements are equivalent.
	\begin{enumerate}
		\item[(i)] $S_g^{n,0}:\tp\to\tq$ is compact.
		
		\item[(ii)] $S_g^{n,0}:\p0\to\q0$ is compact.
		
		\item[(iii)] $
		\lim_{|z|\to1}|g(z)|(1-|z|^2)^{\frac{\b+2}{q}-\frac{\a+2}{p}}=0.
		$
	\end{enumerate}
\end{Theorem}

\begin{Theorem}\label{sn4}
	Let $0<q<p<\infty$, $\a,\b>-2$ with $\frac{p\b-q\a}{p-q}>-2$, $g\in H(\D)$ and $n\in\NN$. Then the following statements are equivalent.
	\begin{enumerate}
		\item[(i)] $S_g^{n,0}:\tp\to\tq$ is compact.
		
		\item[(ii)] $S_g^{n,0}:\p0\to\q0$ is compact.
		
		\item[(iii)] 
		$
		g\in AT^\infty_{\frac{pq}{p-q},0}(\frac{p\b-q\a}{p-q}).
		$
	\end{enumerate}
\end{Theorem}

\end{document}